\documentclass[11pt]{article}

\usepackage{amssymb, amsfonts, amsmath,amsthm, graphicx, bbm, mathrsfs}
\usepackage[margin=1in]{geometry}
\usepackage{color,tikz}
\usepackage[retainorgcmds]{IEEEtrantools}

\DeclareMathOperator{\mdeg}{mdeg}

\DeclareMathOperator{\Pic}{Pic}

\DeclareMathOperator{\Tot}{Tot}

\DeclareMathOperator{\Cone}{Cone}
\DeclareMathOperator{\lcm}{lcm}

\newcommand{\Q}{\mathbb{Q}}
\newcommand{\Z}{\mathbb{Z}}
\newcommand{\R}{\mathbb{R}}
\newcommand{\N}{\mathbb{N}}
\newcommand{\proj}{\mathbb{P}}

\newcommand{\mfrak}{\mathfrak{m}}
\newcommand{\bfa}{\mathbf a}
\newcommand{\bfb}{\mathbf b}
\newcommand{\bfc}{\mathbf c}
\newcommand{\bfe}{\mathbf e}
\newcommand{\bft}{\mathbf t}
\newcommand{\kfield}{\mathbbm k}
\newcommand{\Acal}{\mathcal A}

\newcommand{\pilevar}{\Gamma}
\newcommand{\bfq}{\mathbf q}
\newcommand{\bfone}{\mathbf 1}

\newtheorem{theorem}{Theorem}
\newtheorem{lemma}[theorem]{Lemma}

\newtheorem{proposition}[theorem]{Proposition}

\theoremstyle{remark}
\newtheorem{example}[theorem]{Example}

\title{A duality theorem for syzygies of Veronese ideals}
\author{Stepan Paul\\ California Polytechnic State University\\San Luis Obispo, CA 93407--0403\\(805) 756--5552\\stpaul@calpoly.edu}
\date{}

\begin{document}


\maketitle


\abstract{
We prove a duality theorem for simplicial complexes arising from a combinatorial construction we define, which applies to the squarefree monomial complexes for Veronese ideals of projective spaces and weighted projective spaces. Our theorem yields a formula for the multigraded Betti numbers of these Veronese ideals in terms of the reduced homology groups of these complexes which is dual to one given by Bruns and Herzog. We apply this formula in several ways, including by giving an algorithm for finding the highest syzygy of such a Veronese ideal.
}

\section*{Introduction}

A \emph{Veronese embedding} of a weighted projective space $\proj(\bfq)$ is an embedding defined by a very ample line bundle $\mathscr H^{\otimes d}$ into the projectivization $\proj^N$ of its space of global sections (here $\mathscr H$ is the weighted hyperplane bundle). This paper is motivated by the study of the defining ideal $I_{\bfq,d}$ of the image of this embedding---the \emph{Veronese ideal}---and its quotient $R_{\bfq,d}:=S/I_{\bfq,d}$ via their free resolutions as a module over the homogeneous coordinate ring $S:=\kfield[X_0,\ldots,X_N]$ of $\proj^N$.

Given a finite subset $\Acal\subset\N^{n+1}$ (we use the convention $0\in\N$) and any $\bfc\in\Z^{n+1}$, we define the \emph{pile complex} $\pilevar_\bfc(\Acal)$ to be the abstract simplicial complex
\begin{equation}\label{definecomp}
\pilevar_\bfc(\mathcal A):=\left\{\mathcal S\subseteq\mathcal A\left|\sum_{\bfa\in\mathcal S}\bfa\leq\bfc\right.\right\}.
\end{equation}
Here we are using the partial ordering on $\Z^{n+1}$ given by $\bfa\geq\bfb$ if and only if $\bfa-\bfb\in\N^{n+1}$.

This construction is a variation on Bruns and Herzog's \emph{squarefree divisor complex}, $\Delta_\bfc(\Acal)$, given in~\cite{MR1481087}. In that paper, they show that the multigraded Betti numbers of the toric ideal $I_\Acal$ defined by $\Acal$ are computed by the reduced homology groups of these squarefree divisor complexes via
\begin{equation}\label{bhformula}
\beta_{i,\bfc}(S/I_\Acal)=\dim\tilde H_{i-1}(\Delta_\bfc(\Acal);\kfield)
\end{equation}
for $\bfc$ in the semigroup $\N\Acal$ generated by $\Acal$. In Proposition~\ref{iffprop}, we show that $\Delta_\bfc(\Acal)=\Gamma_\bfc(\Acal)$ for all $\bfc$ precisely when $I_\Acal$ cuts out a Veronese embedding of some weighted projective space or certain finite quotients thereof.

The main result of this paper is Theorem~\ref{intromaintheorem}, which is a purely combinatorial duality statement about pile complexes.


\begin{theorem}\footnote{During the writing and submission process, it was brought to the author's attention that Theorem~\ref{intromaintheorem} was proved previously by X. Dong in~\cite{MR1871691}.}\label{intromaintheorem}
Suppose that $\bfc\in\Z^{n+1}$, and that $\Acal\subset\N^{n+1}$ is a finite set with $N+1$ elements which generates a convex cone in $\R^{n+1}$ containing all of $\N^{n+1}$. Let 
\begin{equation}\label{tcheck}
\bft =\sum_{\bfa\in\Acal}\bfa,\qquad \check\bfc=\bft-\bfc-(1,\ldots,1).
\end{equation}
Then
$$\tilde H_{i-1}(\pilevar_\bfc(\Acal);\kfield)\cong\tilde H_{N-n-i-1}(\pilevar_{\check \bfc}(\Acal);\kfield)^*.$$
\end{theorem}

This result may be interesting to some readers in its own right. For this reason, Section~\ref{combsec} is dedicated to its proof and sits logically independent of its applications to commutative algebra and algebraic geometry. Section~\ref{bgsec} contains some background and references on weighted projective space, Veronese embeddings, and multigraded Betti numbers, and Section~\ref{appsec} will contain applications of the theorem to these topics.

In particular, Section~\ref{appsec} begins with Theorem~\ref{dualformula}, which gives a formula dual to (\ref{bhformula}) in the case that $I_\Acal=I_{\bfq,d}$:
$$ \beta_{i,\bfc}(R_{\bfq,d})=\dim\tilde H_{N-n-i-1}(\pilevar_{\check \bfc}(\Acal);\kfield)\text{ for }\bfc\in\N\Acal.$$



We then go on to prove some immediate consequences of this formula, some of which recover known properties of weighted projective space, and some of which seem to be original results. We summarize by listing our findings here.

\begin{itemize}
\item The projective dimension of $R_{\bfq,d}$ is $N-n$; in particular, $R_{\bfq,d}$ is Cohen-Macaulay (Proposition~\ref{pdim}).
\item An upper bound on the Castelnuovo-Mumford regulary of $R_{\bfq,n}$, for which equality holds in the case of $\proj^n$ (Proposition~\ref{reg}).
\item $R_{\bfq,d}$ is Gorenstein if $dr|q_0+\cdots+q_n$, where $r=\lcm(q_0,\ldots,q_n)$ (Proposition~\ref{gorenstein}).
\item An algorithm for finding the highest nonzero syzygy of $R_{\bfq,n}$ with an explict formula in the case of $\proj^n$ (Theorem~\ref{introlastsyzygy}).
\end{itemize}



\section{Commutative Algebra and Algebraic Geometry Background}\label{bgsec}

In this section, we will review some background necessary to apply Theorem~\ref{intromaintheorem} towards finding Betti numbers of Veronese embeddings of weighted projective space.

\paragraph*{Notation}


Throughout we will use the convention that $\N\ni0$. If $\alpha=(\alpha_0,\ldots,\alpha_N)\in\Z^{N+1}$, then we will sometimes abbreviate the monomial $X_0^{\alpha_0}\cdots X_N^{\alpha_N}$ as $X^\alpha$. If $\Acal$ is any subset of $\Z^{n+1}$, we will write
\begin{itemize}
\item $\Z\Acal$ for the subgroup of $\Z^{n+1}$ generated by $\Acal$,
\item $\N\Acal$ for the semigroup of $\N$--linear combinations of elements of $\Acal$, and
\item $\Cone(\Acal)$ for the convex cone of positive $\R$--linear combinations of elements of $\Acal$.
\end{itemize}

We will use $\bfe_0,\ldots,\bfe_n$ for the standard basis elements of $\Z^{n+1}$. 
For elements $\bfa,\bfb\in\R^{n+1}$, we write $\bfa\cdot\bfb$ for the standard dot product $a_0b_0+\cdots+a_nb_n$. We will also find it convenient to define $\bfone:=(1,\ldots,1)\in\Z^{n+1}$.

\paragraph*{Toric Ideals}

The world of toric varieties has long been a testing ground for algebraic geometers due to their combinatorial description and concrete structure. We cite some basic definitions and results which can be found in, for example,~\cite{MR2810322}.  Given a finite subset $\mathcal A=\{\bfa_0,\ldots,\bfa_N\}$ of $\N^{n+1}$, we define a semigroup homomorphism
$$\begin{array}{rccc}
\Phi_{\mathcal A}:&\N^{N+1}&\longrightarrow&\N^{n+1}\\
&(x_0,\ldots,x_N)&\longmapsto&x_0\bfa_0+\cdots+x_N\bfa_N.
\end{array}$$
Then $\Phi_\Acal$ induces a semigroup algebra morphism
$$\begin{array}{rccc}
\hat\Phi_{\mathcal A}:&\kfield[X_0,\ldots,X_N]&\longrightarrow&\kfield[Y_0,\ldots,Y_n]\\
&X_i&\longmapsto&Y^{\mathbf a_i}.
\end{array}$$
The \emph{toric ideal} $I_\Acal$ of $S:=\kfield[X_0,\ldots,X_N]$ 
 corresponding to $\mathcal A$ is defined to be
the kernel of $\hat\Phi_\Acal$. Notice that
$$I_{\mathcal A}=\left\langle\left. X^\alpha- X^\beta\right|\alpha,\beta\in\N^{N+1},\;\Phi_{\Acal}(\alpha)=\Phi_\Acal(\beta)\right\rangle.$$

If $\mathcal A$ is \emph{homogeneous}---that is, if there exists $\omega\in\Q^{n+1}$ such that $\omega\cdot\bfa=1$ for all $\bfa\in\Acal$---then the vanishing set of $I_{\mathcal A}$ in $\proj^N$ is a \emph{projective toric variety}. 

A toric ideal has a natural multigrading by $\Z^{n+1}$ given by $\mdeg(X_i)=\bfa_i$. This multigraded structure is inherited by the polynomial ring $S$, the quotient ring $R_{\mathcal A}:=S/I_{\mathcal A}$, and all higher syzygies of $I_{\mathcal A}$ (as $S$-modules). If $\Acal$ is homogeneous, then we can recover the usual $\Z$--grading on $S$, given by $\deg(X_i)=1$, in a consistent way via $\deg(X^\alpha)=\omega\cdot\mdeg(X^\alpha)$.

\paragraph*{Veronese Embeddings}
In general, a very ample line bundle $\mathscr L$ on a projective variety $X$ defines an embedding $\nu_{\mathscr L}$ of $X$ into some projective space. If $X$ admits a very ample line bundle $\mathscr H$ which generates the Picard group $\Pic(X)$, then we refer to $\nu_d:=\nu_{\mathscr H^{\otimes d}}$ as the \emph{$d$th Veronese embedding} of $X$.

We remark that $X$ is of course isomorphic to the image of each of its Veronese embeddings. The variety $X$ may already be well-understood as an abstract variety, but we are interested in studying the images of these embeddings as subvarieties of projective space by investigating the defining ideals of these images.

\paragraph*{Weighted Projective Space}

We say that $\bfq=(q_0,\ldots,q_n)\in\N^{n+1}$ is a \emph{weights vector} if $\gcd(q_0,\ldots,q_n)$. The \emph{weighted projective space} $\proj(\bfq)$ is then defined to be the quotient of $\kfield^{n+1}\smallsetminus\{0\}$ under the $\kfield^*$--action 
$$\lambda.(x_0,\ldots,x_n)=(\lambda^{q_0}x_0,\ldots,\lambda^{q_n}x_n).$$ 
Notice that $\proj(\bfone)=\proj^n$.

The homogeneous coordinate ring of $\proj(\bfq)$ is $\kfield[Y_0,\ldots,Y_n]$ with the weighted $\Z$--grading $\deg(Y_i)=q_i$. We define the \emph{weighted hyperplane bundle} $\mathscr H$ on $\proj(\bfq)$ to be the line bundle corresponding to the class of divisors cut out by weighted homogeneous polynomials of weighted degree $r$ where $r:=\lcm(q_0,\ldots,q_n)$. In the case of $\proj^n$, $\mathscr H$ is the usual hyperplane bundle. The weighted hyperplane bundle does in fact generate $\Pic(\proj(\bfq))$, and so we can say $\mathscr H^{\otimes d}$ yields the $d$th Veronese embedding of $\proj(\bfq)$.

We can understand the $d$th Veronese embedding of $\proj(\bfq)$ by thinking of weighted projective space as a projective toric variety. Consider the very ample convex polytope
$$P_\bfq:=\left\{\left.\mathbf x\in\R^{n+1}\right|\bfq\cdot\mathbf x=r,\;x_i\geq0\;\forall\; i \right\}$$
Notice that $P_\bfq$ is the convex hull in $\R^{n+1}$ of the points $(r/q_0)\bfe_0,\ldots,(r/q_n)\bfe_n$. Let $\Acal$ be the set of integral points of $dP_\bfq$, so that $\Acal$ is the set of $\bfa$ such that $X^\bfa$ has weighted degree $dr$. Notice that $\Acal$ is homogeneous since $\omega\cdot\bfa=1$ for all $\bfa\in\Acal$ where
\begin{equation}\label{omega}
\omega=\left(\frac{q_0}{dr},\ldots,\frac{q_n}{dr}\right).
\end{equation}

The projective toric variety cut out by $I_\Acal$ is the $d$th Veronese embedding of $\proj(\bfq)$. We will refer to $P_\bfq$ as the \emph{standard defining polytope} of $\proj(\bfq)$. 

For proofs of these facts and a nice overview of weighted projective space as a toric variety, one may refer to~\cite{2011arXiv1112.1677R}.




\paragraph*{Multigraded Betti Numbers}

Suppose $M$ is any $S$--module with a $\Z^{n+1}$--multigraded structure. A \emph{multigraded free resolution} of $M$ is a multigraded $S$--complex of free modules
$$F_\bullet:\;\cdots \longrightarrow F_2\stackrel{d_2}\longrightarrow F_1\stackrel{d_1}\longrightarrow F_0\longrightarrow0$$
whose only nonzero homology is $H_0(F_\bullet)=M$. In general, an $S$--complex $C_\bullet$ is said to be \emph{minimal} if the image of $C_i\rightarrow C_{i-1}$ is contained in $\mathfrak mC_{i-1}$, where $\mathfrak m=\langle X_0,\ldots,X_N\rangle$ is the irrelevant ideal of $S$. Up to isomorphism, any finitely generated multigraded $S$--module $M$ has a unique minimal multigraded free resolution, and it is a direct summand of any free resolution of $M$.

Suppose $M$ is finitely generated, and $F_\bullet$ is its minimal multigraded free resolution. Then the \emph{$i$th syzygy} $\Omega_i$ of $M$ is the image of the map $d_i:F_i\rightarrow F_{i-1}$ for $i\geq 1$; the $0$th syzygy $\Omega_0$ is defined to be $M$. For $i\in\N$, $\bfc\in \Z^{n+1}$, let $\beta_{i,\bfc}(M)$ be the number of minimal multidegree--$\bfc$ generators of $\Omega_i$. Then the minimality of $F_\bullet$ implies that
$$F_i\cong\bigoplus_{\bfc\in\Z^{n+1}} S(-\bfc)^{\beta_{i,\bfc}(M)}.$$
Here $S(-\bfc)$ is the multigraded rank--$1$ free $S$--module with the degree shift $S(-\bfc)_{\bfb} = S_{\bfb-\bfc}$. The $\beta_{i,\bfc}(M)$ are called the \emph{multigraded Betti numbers} of $M$, and are invariants of the module. The collection of all such Betti numbers is called the \emph{Betti table} of $M$.

If $\Acal\subset\N^{n+1}$ is a finite homogeneous set, then the $\Z$--graded Betti numbers $\beta_{i,j}(R_\Acal)$ (those corresponding to the usual $\Z$--grading on $S$) can be recovered from the $\Z^{n+1}$--multigraded Betti numbers via
\begin{equation}\label{bettisum}
\beta_{i,j}(R_\Acal)=\sum_{\omega\cdot\bfc=j}\beta_{i,\bfc}(R_\Acal).
\end{equation}

\paragraph*{Squarefree Divisor Complexes}

In~\cite{MR1481087}, Bruns and Herzog define the \emph{squarefree divisor complex} on a finite subset $\Acal\subset\N^{n+1}$ for some $\bfc\in\Z\Acal$ to be the abstract simplicial complex on the set $\Acal$ given by
\begin{equation}\label{bhcomplex}
\Delta_\bfc(\Acal):=\left\{\mathcal S\subseteq\Acal\left|\bfc-\sum_{\bfa\in\mathcal S}\bfa\in\N\Acal\right.\right\}.
\end{equation}

The reduced homology groups of $\Delta_\bfc(\Acal)$ can then be used to find the multigraded Betti numbers of the toric module $R_\Acal$ as follows.

\begin{theorem}[Bruns, Herzog]\label{bhtheorem}
For any finite $\Acal\subset\N^{n+1}$, $\bfc\in\Z\Acal$,
$$\beta_{i,\bfc}(R_\Acal)=\dim\tilde H_{i-1}(\Delta_\bfc(\Acal);\kfield).$$
\end{theorem}

Comparing (\ref{definecomp}) and (\ref{bhcomplex}), it is clear that in general $\Delta_\bfc(\Acal)$ is a sub-complex of the pile complex $\pilevar_\bfc(\Acal)$ for $\bfc\in\N\Acal$ since $\N\Acal\subseteq\N^{n+1}$. However, the following set of propositions show that equality holds when $\Acal$ defines a Veronese embedding of some weighted projective space.

\begin{proposition}\label{samecomp}
Suppose $\Acal$ is a finite homogeneous subset of $\N^{n+1}$ spanning $\R^{n+1}$ as a vector space. Then the following are equivalent:
\begin{enumerate}
\item\label{sat} $\N\Acal=\Z\Acal\cap\N^{n+1}$.
\item\label{sat2} $\N\Acal$ is saturated in $\Z\Acal$, and $\Cone(\Acal)=\Cone(\N^{n+1})$.
\item\label{sublat} For some $d\geq1$ and weights vector $\bfq$, $\Acal$ is the intersection of $dP_\bfq$ with a sublattice $L\leq\Z^{n+1}$ containing the vertices of $dP_\bfq$.
\item\label{eq} $\Delta_\bfc(\Acal)=\Gamma_\bfc(\Acal)$ for all $\bfc\in\Z\Acal$.
\end{enumerate}
\end{proposition}

\begin{proof}
We begin by showing (\ref{sat} $\Rightarrow$ \ref{sat2}). Suppose $\N\Acal=\Z\Acal\cap\N^{n+1}$. If $\bfb\in\Z\Acal$  and $k\bfb\in\N\Acal$ for some $k\geq0$, then $\bfb\in\N^{n+1}$, so $\bfb\in\N\Acal$, proving saturation. Also, since $\Acal$ spans $\R^{n+1}$ by assumption, $\Cone{\Z\Acal}=\R^{n+1}$. Hence,
$$\Cone(\Acal)=\Cone(\N\Acal)=\Cone(\Z\Acal)\cap\Cone(\N^{n+1})=\Cone(\N^{n+1}).$$

Next we show (\ref{sat2} $\Rightarrow$ \ref{sublat}). Assuming $\Cone(\Acal)=\Cone(\N^{n+1})$, we know there exist positive integers $s_0,\ldots,s_n$ such that $s_0\bfe_0,\ldots,s_n\bfe_n\in\Acal$. Let 
$$d=\gcd(s_0,\ldots,s_n),\;r=\lcm\left(\frac{s_0}{d},\ldots,\frac{s_n}{d}\right),\;\bfq=\left(\frac{dr}{s_0},\ldots,\frac{dr}{s_n}\right).$$
Since $\Acal$ is homogeneous, it is contained in the hyperplane containing $s\bfe_0,\ldots,s\bfe_n$; $dP_\bfq$ is the intersection of this hyperplane with the first orthant, so $\Acal\subset dP_\bfq$. Saturation then implies $\Acal=dP_\bfq\cap\Z\Acal$.

To show (\ref{sublat} $\Rightarrow$ \ref{eq}), we need only show that $\pilevar_\bfc(\Acal)\subseteq\Delta_\bfc(\Acal)$ since the opposite inclusion holds by definition. Suppose $\mathcal S\in\pilevar_\bfc(\Acal)$ for some $\bfc\in\Z\Acal$. Then if $\bfb=\bfc-\sum_{\bfa\in\mathcal S}\bfa$, we know $\bfb\in\N^{n+1}$ by definition and $\bfb\in (kdP_\bfq)\cap L$ for some $k\in\Z$. Since $dP_\bfq$ is very ample, this implies $\bfb\in k(dP_\bfq\cap L)$, so $\bfb\in\N\Acal$. Hence $\mathcal S\in\Delta_\bfc(\Acal)$.


Finally, suppose there exists $\bfb\in(\Z\Acal\cap\N^{n+1})\smallsetminus\N\Acal$, and assume $\bfb$ is maximal with this property. Then $\bfb+\bfa'\in\N\Acal$ for some $\bfa'\in\Acal$. But then $\{\bfa'\}\in\pilevar_{\bfb+\bfa'}(\Acal)\smallsetminus\Delta_{\bfb+\bfa'}(\Acal)$. Hence, we have shown the contrapositive of (\ref{eq} $\Rightarrow$ \ref{sat}).
\end{proof}

\begin{proposition}\label{iffprop}
The set $\Acal=dP_\bfq\cap\Z^{n+1}$ satisfies the conditions of Proposition~\ref{samecomp} for any $d\geq1$ and weights vector $\bfq$.

Furthermore, if $\Acal$ satifies the conditions of Proposition~\ref{samecomp}, then $I_\Acal$ defines the quotient of a weighted projective space by a finite group action.
\end{proposition}

\begin{proof}
The first assertion follows immediately from the third equivalent condition in Proposition~\ref{samecomp}.

To prove the second assertion, we again use the third condition, and notice that $dP_\bfq$ is an $n$-simplex, so taking the fan $\Sigma_P$ (with respect to the appropriate lattice), we can apply Theorem~2.11 of~\cite{MR1290195}.
\end{proof}

We remark that the applications of Theorem~\ref{intromaintheorem} can be extended to toric embeddings of finite quotients of weighted projective space arising as in the proof of Proposition~\ref{iffprop}. We do not attempt to characterize such embeddings, but we offer the following example to show that the generalization is nontrivial.

\begin{example}
Let
$$\Acal=\{(4,0,0),(2,2,0),(1,1,1),(0,4,0),(0,0,2)\}.$$
Notice that $\Acal$ satisfies the conditions of Proposition~\ref{samecomp}, but it is the intersection of a convex set with an index four sublattice $L\leq\Z^{n+1}$. The ray generators of the fan $\Sigma_P$ in the dual lattice $L^\vee$ generate an index two sublattice of $L^\vee$, and so the projective toric variety cut out by $I_\Acal$ is a quotient of $\proj(1,1,2)$ by $\Z/2\Z$.
\end{example}

\section{Proof of Theorem~\ref{intromaintheorem}}\label{combsec}

In this section we will prove Theorem~\ref{intromaintheorem}, independent of the applications in Section~\ref{appsec}.

As in the introduction, suppose $\mathcal A$ is a finite subset of $\N^{n+1}$ and $\mathbf c\in\Z^{n+1}$. Define $\Gamma_\bfc(\Acal)$ as in (\ref{definecomp}), and $\bft$ and $\check\bfc$ as in (\ref{tcheck}).

Notice first that $\mathcal A\subset\N^{n+1}$ guarantees that $\pilevar_{\mathbf c}(\mathcal A)$ is indeed a complex. Indeed, if $\sum_{\bfa\in \mathcal S}\bfa\leq\bfc$, then for any $\mathcal T\subseteq\mathcal S$, $\sum_{\bfa\in\mathcal T}\bfa\leq\bfc$. 



From now on, we assume all homology groups are taken over the base field $\mathbbm k$.


\begin{lemma}\label{starlemma}
Assume $\mathcal A$ satisfies the conditions of Theorem~\ref{intromaintheorem}. Suppose $\mathbf b=(b_0,\ldots,b_n)$ is a point in $\Z^{n+1}$ such that $b_j=t_j$ for some $j$ (where $\bft=(t_0,\ldots,t_n)$). Then $\tilde H_i(\pilevar_{\mathbf b}(\mathcal A))=0$ for all $i$.
\end{lemma}

This Lemma can be deduced from Proposition~3.1 of~\cite{MR1481087}, but we provide a self-contained proof.

\begin{proof}
Without loss of generality, assume $j=0$ so that $\mathbf b = (t_0,b_1,\ldots,b_n)$. In order to generate a cone containing all of $\N^{n+1}$, $\mathcal A$ must contain an element $\bfa'=d\bfe_0$ for some $d>0$.

We claim that $\mathbf a'$ is an element of any maximal face of $\pilevar_{\mathbf b}$. Indeed, if $\mathcal S$ is any face of $\pilevar_{\mathbf b}$,
$$\sum_{\mathbf a\in\mathcal S\cup\{\mathbf a'\}}a_0 \leq t_0;\qquad \sum_{\mathbf a\in\mathcal S\cup\{\mathbf a'\}}a_i = \sum_{\mathbf a\in\mathcal S}a_i\leq b_i,\quad i=1,\ldots,n.$$
That is,
$$\mathcal S\in\pilevar_{\mathbf b}\Longrightarrow\mathcal S\cup\{\mathbf a'\}\in\pilevar_{\mathbf b}(\Acal).$$
Because of this property, $\pilevar_{\mathbf b}(\Acal)$ is star-shaped with center $\mathbf a'$, and hence contractible.
\end{proof}

We introduce one more bit of background before the proof. For any abstract simplicial complex $\Delta$ on a set $\mathcal B$, we let $\Delta^*$ be its \emph{Alexander dual}---the simplicial complex of complements of non-faces of $\Delta$. If $M$ is the number of elements in $\mathcal B$, there is an isomorphism
$$\tilde{H}_i(\Delta^*)\cong\tilde{H}^{M-i-3}(\Delta).$$
See~\cite{MR2556456} for a proof.

\begin{proof}[Proof of Theorem~\ref{intromaintheorem}]
Let $\mathcal A$ and $\mathbf c = (c_0,\ldots,c_n)$ be as in the theorem. Also, define $\mathbf t=(t_0,\ldots,t_n)$ and $\mathbf{\check c}$ as in (\ref{tcheck}).

We can describe the faces $\mathcal S\in\pilevar_{\mathbf c}^*$ explicitly as follows.
\begin{eqnarray*}
\mathcal S\in\pilevar_{\mathbf c}^*&\Leftrightarrow&\mathcal A\smallsetminus\mathcal S\notin\pilevar_{\mathbf c}\\
&\Leftrightarrow&\sum_{\mathbf a\notin \mathcal S} a_i> c_i\quad\text{for some $i$}\\
&\Leftrightarrow&\sum_{\mathbf a\in \mathcal S} a_i< t_i-c_i\quad\text{for some $i$}\\
&\Leftrightarrow&\sum_{\mathbf a\in \mathcal S} a_i\leq t-c_i-1\quad\text{for some $i$}\\
&\Leftrightarrow&\mathcal S\in\pilevar_{(t_0-c_0-1,t_1,\ldots,t_n)}\cup\cdots\cup\pilevar_{(t_0,\ldots,t_{n-1},t_n-c_n-1)}.
\end{eqnarray*}

Now let $\mathcal R=\{0,\ldots,n\}$ and for any non-empty $\sigma\subseteq\mathcal R$, let $\mathbf c^\sigma\in\N^{n+1}$ be defined by
$$c_j^\sigma=\begin{cases}t_j-c_j-1&j\in\sigma\\t_j&\text{else}\end{cases}$$
Then let $\pilevar^j=\pilevar_{\mathbf c^{\{j\}}}$ for $j\in\mathcal R$, and $\pilevar^\sigma=\pilevar_{\mathbf c^\sigma}$. Notice that $\pilevar^\sigma=\bigcap_{j\in\sigma}\pilevar^j$ and $\pilevar^{\mathcal R}=\pilevar_{\mathbf{\check c}}$. We also observe that $\pilevar_{\mathbf c}^*=\bigcup_{j=0}^n\pilevar^j$. We let $\pilevar_{\mathbf c}^*[i]$ be the set of $i$--simplices of $\pilevar_{\mathbf c}^*$.

We will use a generalized Mayer-Vietoris sequence to show that
$$\tilde H_{i+n}(\pilevar_{\mathbf c}^*)=\tilde H_{i}\left(\pilevar_{\mathbf{\check c}}\right).$$

Define a double complex $C_{\bullet,\bullet}$ of reduced simplicial chain groups with 
$$C_{p,q}=\bigoplus_{\sigma\in\pilevar_{\mathbf c}^*[p]}\tilde C_q(\pilevar^\sigma)$$
as shown.
$$\begin{array}{ccccccccccc}
&&\vdots&&\vdots&&&&\vdots&\\
&&\downarrow&&\downarrow&&&&\downarrow&\\
0&\rightarrow&\tilde C_1\left(\pilevar_{\mathbf{\check c}}\right)&\stackrel\delta\rightarrow&\bigoplus_{j=0}^n\tilde C_1(\pilevar^{\mathcal R\smallsetminus j}) &\rightarrow&\cdots&\rightarrow&\bigoplus_{j=0}^n\tilde C_1(\pilevar^j)&\rightarrow&0\\
&&\partial\downarrow&&\partial\downarrow&&&&\partial\downarrow\\
0&\rightarrow&\tilde C_0\left(\pilevar_{\mathbf{\check c}}\right)&\stackrel\delta\rightarrow&\bigoplus_{j=0}^n\tilde C_0(\pilevar^{\mathcal R\smallsetminus j}) &\rightarrow&\cdots&\rightarrow&\bigoplus_{j=0}^n\tilde C_0(\pilevar^j)&\rightarrow&0\\
&&\epsilon\downarrow&&\epsilon\downarrow&&&&\epsilon\downarrow\\
0&\rightarrow&\tilde C_{-1}(\pilevar_{\mathbf{\check c}})&\stackrel\delta\rightarrow&\bigoplus_{j=0}^n\tilde C_{-1}(\pilevar^{\mathcal R\smallsetminus j})&\rightarrow&\cdots&\rightarrow&\bigoplus_{j=0}^n\tilde C_{-1}(\pilevar^j)&
\rightarrow&0\\
&&\downarrow&&\downarrow&&&&\downarrow\\
&&0&&0&&&&0&
\end{array}$$

Each of the columns is the direct sum of augmented simplicial chain complexes with the usual boundary operator $\partial$. The row differentials $\delta$ are given by alternating inclusions.

As shown in Section~VII.4 of~\cite{MR672956}, the spectral sequence starting with horizontal homology degenerates at the $E^2$ page to show that we have
$$H_q(\Tot(C_{\bullet,\bullet}))=\tilde H_q(\pilevar_{\mathbf c}^*).$$

Now we take the spectral sequence starting with vertical homology, and we get
$$E^1_{p,q}=\bigoplus_{\sigma\in \pilevar_{\mathbf c}^*[p]}\tilde H_q(\pilevar^\sigma).$$

Now, if $\sigma$ is a non-empty proper subset of $\mathcal R$, then $c^\sigma_j=t_j$ for some $j$. Hence by Lemma~\ref{starlemma}, $\pilevar^\sigma$ has trivial homologies for such $\sigma$. 
Therefore, for $q\geq0$, we really have
$$E^1_{p,q}=\begin{cases}\tilde H_q\left(\pilevar_{\mathbf{\check c}}\right)&p=n\\0&\text{else} \end{cases}$$

Therefore the spectral sequence degenerates on the first page, and we have $\tilde H_{q+n}(\pilevar_{\mathbf c}^*)=\tilde H_{q}\left(\pilevar_{\mathbf{\check c}}\right)$ as claimed. 


Putting this all together, we have
\begin{IEEEeqnarray*}{+rCll+x*}
\tilde H_{i-1}(\pilevar_{\mathbf c})&\cong&\tilde H^{N-i-1}(\pilevar_{\mathbf c}^*)&\text{(Alexander duality)}\\
&\cong&\tilde H_{N-i-1}(\pilevar^*_{\mathbf c})^*&\text{(universal coefficient theorem)}\\
&\cong&\tilde H_{N-n-i-1}\left(\pilevar_{\mathbf{\check c}}\right)^*&\text{(Mayer-Vietoris)}&\qedhere
\end{IEEEeqnarray*}
\end{proof}

\section{Application to Betti Numbers}\label{appsec}


We begin by putting together the results accumulated thus far.

\begin{theorem}\label{dualformula}
Suppose $\Acal\subset\N^{n+1}$ is a homogeneous finite set satisfying the conditions of Proposition~\ref{samecomp}. %
Then for any $\bfc\in\Z\Acal$, we have the following formula for the Betti numbers of $R_\Acal$:
$$\beta_{i,\bfc}(R_\Acal)=\dim\tilde H_{N-n-i-1}(\Gamma_{\check\bfc}(\Acal)).$$
In particular, this formula holds when $R_\Acal$ is a Veronese module of a weighted projective space.
\end{theorem}

\begin{proof}
Putting together Theorem~\ref{intromaintheorem}, Bruns and Herzog's Theorem~\ref{bhtheorem}, and Proposition~\ref{samecomp}, for any $\bfc\in\Z\Acal$ we have
$$\beta_{i,\bfc}=\dim\tilde H_{i-1}(\Delta_\bfc(\Acal))=\dim\tilde H_{i-1}(\Gamma_\bfc(\Acal))=\dim\tilde H_{N-n-i-1}(\Gamma_{\check\bfc}(\Acal)).$$
The final assertion holds because of Proposition~\ref{iffprop}.
\end{proof}

\begin{example}
Let $\bfq=(1,1,1)$ and $d=2$. In this case, $I_{(1,1,1),2}$ defines the second Veronese embedding of $\proj^2$, known as the \emph{Veronese surface}.

We have
$$\Acal=\{(2,0,0),(1,1,0),(1,0,1),(0,2,0),(0,1,1),(0,0,2)\}.$$
Here $N=5$, and $\bft=(4,4,4)$. Letting $\bfc=(2,2,2)$, we have $\check\bfc=(1,1,1)$. Notice that $\bfc$ is in $\Z\Acal$ but $\check\bfc$ is not. We record the maximal faces of the two pile complexes below.
$$\begin{array}{ccc}
\underline{\pilevar_{(2,2,2)}(\Acal)}&\qquad&\underline{\pilevar_{(1,1,1)}(\Acal)}\\
\{(2,0,0),(0,1,1)\}&&\{(1,1,0)\}\\
\{(2,0,0),(0,2,0),(0,0,2)\}&&\{(1,0,1)\}\\
\{(1,1,0),(1,0,1),(0,1,1)\}&&\{(0,0,1)\}\\
\{(1,1,0),(0,0,2)\},\\
\{(1,0,1),(0,2,0)\}
\end{array}$$

Below we illustrate a geometric realization of each pile complex.

{\center
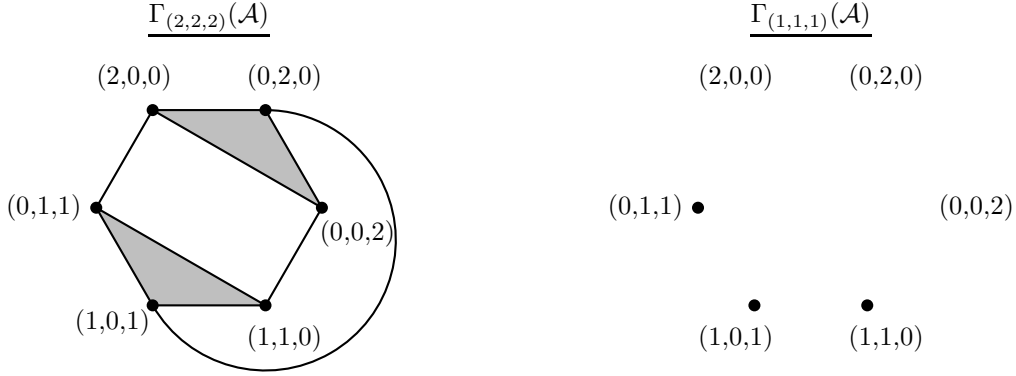
\begin{tikzpicture}[scale = 1]
\draw (120:2) node {(2,0,0)};
\draw (60:2) node {(0,2,0)};
\draw (350:2) node {(0,0,2)};
\draw (300:2) node {(1,1,0)};
\draw (230:2) node {(1,0,1)};
\draw (180:2.2) node {(0,1,1)};
\filldraw[fill = gray!50, thick] (0:1.5) -- (60:1.5) -- (120:1.5) -- (0:1.5);
\filldraw[fill = gray!50, thick] (180:1.5) -- (240:1.5) -- (300:1.5) -- (180:1.5);
\draw[thick] (120:1.5) -- (180:1.5);
\draw[thick] (0:1.5) -- (300:1.5);
\fill (0:1.5) circle (.08);
\fill (60:1.5) circle (.08);
\fill (120:1.5) circle (.08);
\fill (180:1.5) circle (.08);
\fill (240:1.5) circle (.08);
\fill (300:1.5) circle (.08);
\draw[thick] (240:1.5) arc (210:450:1.732);
\draw (90:2.5) node {$\underline{\pilevar_{(2,2,2)}(\Acal)}$};

\begin{scope}[shift = {(8,0)}]
\draw (120:2) node {(2,0,0)};
\draw (60:2) node {(0,2,0)};
\draw (0:2.2) node {(0,0,2)};
\draw (300:2) node {(1,1,0)};
\draw (240:2) node {(1,0,1)};
\draw (180:2.2) node {(0,1,1)};
\fill (180:1.5) circle (.08);
\fill (240:1.5) circle (.08);
\fill (300:1.5) circle (.08);
\draw (90:2.5) node {$\underline{\pilevar_{(1,1,1)}(\Acal)}$};
\end{scope}
\end{tikzpicture}

}

As predicted by Theorem~\ref{intromaintheorem},
$$\dim\tilde H_1(\pilevar_{(2,2,2)}(\Acal))=\dim\tilde H_0(\pilevar_{(1,1,1)})=2.$$
Hence, 
$$\beta_{2,(2,2,2)}(R_{(1,1,1),2})=\dim\tilde H_0(\pilevar_{(1,1,1)}(\Acal))$$
as predicted by Theorem~\ref{dualformula}.
\end{example}

We think of Theorem~\ref{dualformula} as providing a dual formula to the one offered by Bruns and Herzog's Theorem~\ref{bhtheorem} in the case of Veronese embeddings of weighted projective space. For the purposes of computing Betti numbers, the formulas complement each other well. Bruns and Herzog's formula can more easily calculate \emph{lower} Betti numbers ($\beta_{i,\bfc}$ for small $i$), whereas Theorem~\ref{dualformula} can more easily calculate \emph{higher} Betti numbers ($\beta_{i,\bfc}$ for large $i$).

In this spirit, some of our applications of Theorem~\ref{dualformula} use properties of the lowest reduced homology groups to prove facts about the highest syzygy. We remind the reader of the subtlties of working with reduced homology.
$$\begin{array}{ll}
\tilde H_{-1}(\{\varnothing\};\mathbbm k)=\mathbbm k,& \tilde H_i(\varnothing;\mathbbm k)=0\text{ for all $i$},\\
\tilde H_i(\{\varnothing\};\mathbbm k)=0\text{ for $i\neq-1$},\quad&\dim\tilde H_0(X;\kfield)=\#\{\text{components of $X$}\}-1.
\end{array}$$
By definition, if $\bfc\ngeq\mathbf0$, then $\pilevar_\bfc(\Acal)=\varnothing$, and if $\bfc\geq\mathbf0$ but $\bfc\ngeq\bfa$ for all $\bfa\in\Acal$, then $\pilevar_\bfc(\Acal)=\{\varnothing\}$. Hence we have
\begin{equation}\label{reducedfact}
\tilde H_{-1}(\pilevar_\bfc(\Acal))=\begin{cases}\kfield&\bfc\in\N^{n+1},\text{ and }\bfc\ngeq\bfa\;\forall\;\bfa\in\Acal\\0&\text{else}\end{cases}.
\end{equation}


The results about weighted projective space that follow come as immediate applications of Theorem~\ref{dualformula}.


The projective dimension of an $S$--module is the length of its minimal free resolution---that is, the largest $i$ such that $\beta_{i,\bfc}\neq0$ for some $\bfc$.
\begin{proposition}\label{pdim}
The projective dimension of the $R_{\bfq,d}$ is $N-n$. In particular, $R_{\bfq,d}$ is Cohen-Macaulay.
\end{proposition}
\begin{proof}
Let $p$ be the projective dimension of $R_{\bfq,d}$. As usual, set $\Acal=P_\bfq\cap\Z^{n+1}$. If $i>N-n$, then $N-n-i-1\leq-2$, so for any $\bfc\in\N\Acal$, 
$$\beta_{i,\bfc}=\dim\tilde H_{N-n-i-1}(\Gamma_{\check\bfc}(\Acal))=0.$$
 Hence $p\leq N-n$.

Now, let
$$\mathcal B:=\{\check\bfc\;|\;\bfc\in\N\Acal,\;\pilevar_{\check\bfc}(\Acal)\neq\varnothing\}=\{\bfb\in\N^{n+1}\;|\;\check\bfb\in\N\Acal\}.$$
Let $\bfb\in\mathcal B$ be such that $\bfq\cdot\bfb$ is minimal, so that for any $\bfa\in\Acal$, we know $\bfb-\bfa\notin\mathcal B$. Since $\check\bfb+\bfa\in\N\Acal$, this means $\bfb-\bfa\notin\N^{n+1}$ for all $\bfa\in\N\Acal$. Hence, we use (\ref{reducedfact}) to get
$$\beta_{N-n,\check\bfb}=\dim\tilde H_{-1}(\pilevar_\bfb(\Acal))\neq0.$$
This implies that $p\geq N+n$, proving the proposition.

The last assertion follows from $p+n=N$ since $n=\dim R_{\bfq,n}$ and $N=\dim S$.
\end{proof}

Let $M$ be a finitely generated $S$--module with the usual $\Z$--grading, and let $F_\bullet$ be its minimal $\Z$--graded free resolution. The \emph{Castelnuovo-Mumford regularity} (or just \emph{regularity}) of $M$ is the supremum of values $j-i$ such that $F_i$ has a minimal generator of degree $j$. Using (\ref{omega}), we see that the usual $\Z$--grading on $R_{\bfq,d}$ can be recovered from the $\Z^{n+1}$--multigrading via
$$\deg(X^\alpha)=\frac{\bfq\cdot\mdeg(X^\alpha)}{dr}.$$

In what follows, let $\rho$ be the remainder of $-q_0-\cdots-q_n$ upon division by $dr$ so that
\begin{equation}\label{rhodef}
dr>\rho\geq0,\quad \left\lceil\frac{q_0+\cdots+q_n}{dr}\right\rceil dr=q_0+\cdots+q_n+\rho
\end{equation}
%

\begin{proposition}\label{reg}
The Castelnuovo-Mumford regularity $\delta$ of $R_{\bfq,d}$ with respect to the usual $\Z$--grading on $S$ satisfies
\begin{equation}\label{regbound}
\delta\leq n+1-\left\lceil\frac{q_0+\cdots+q_n}{dr}\right\rceil.
\end{equation}
Equality holds if there exists $\bfb\in\N^{n+1}$ with $\bfq\cdot\bfb=\rho$. 

In particular for the case of $\proj^n$, when $\bfq=\bfone$, we have
$$\delta=n+1-\left\lceil\frac{n+1}{d}\right\rceil.$$
\end{proposition}
\begin{proof}
First observe that for any $\bfc\in\Z^{n+1}$
\begin{equation}\label{qccheck}
\bfq\cdot\check\bfc=\bfq\cdot\bft-\bfq\cdot\bfc-\bfq\cdot\bfone=(N+1)dr-jdr-q_0-\cdots-q_n.
\end{equation}
Suppose the $\Z$--graded Betti number $\beta_{i,j}\neq0$ for some $i,j\in\Z$. Then by~(\ref{bettisum}), there exists $\bfc\in\Z\Acal$ such that $\bfq\cdot\bfc=jdr$ with $\beta_{i,\bfc}\neq0$. By~(\ref{qccheck}), we have
$$\bfq\cdot\check\bfc=(N+1)dr-jdr-q_0-\cdots-q_n.$$
Since $\beta_{i,\bfc}=\dim\tilde H_{N-n-i-1}(\pilevar_{\check\bfc}(\Acal))\neq0$, then $\pilevar_{\check\bfc}(\Acal)$ must contain at least one reduced $(N-n-i-1)$--face. Hence, $\check\bfc$ is at least as large as the sum of $N-n-i$ elements of $\Acal$. Hence $\bfq\cdot\check\bfc\geq (N-n-i)dr$, so we can write
\begin{eqnarray*}
(N+1)dr-jdr-q_0-\cdots-q_n&\geq&(N-n-i)dr\\
n+1-\frac{q_0+\cdots+q_n}{dr}&\geq&j-i
\end{eqnarray*}
proving~(\ref{regbound}).


To prove the assertion about equality, assume $\bfb\in\N^{n+1}$ with $\bfq\cdot\bfb=\rho$. Then using~(\ref{qccheck}) we have
\begin{eqnarray*}
\bfq\cdot\check\bfb&=&(N+1)dr-q_0-\cdots-q_n-\rho\\
\frac{\bfq\cdot\check\bfb}{dr}&=&N+1-\left\lceil\frac{q_0+\cdots+q_n}{dr}\right\rceil
\end{eqnarray*}
Hence $\check\bfb\in\Z\Acal$, and so $\beta_{N-n,\check\bfb}=\dim\tilde H_{-1}(\pilevar_\bfb(\Acal))$. Since $\bfq\cdot\bfb=\rho<dr=\bfq\cdot\bfa$ for all $\bfa\in\Acal$, we know $\bfb\ngeq\bfa$ for all $\bfa\in\Acal$. Hence by~(\ref{reducedfact}), $\beta_{N-n,\check\bfb}=1$, and so
$$\delta\geq N+1-\left\lceil\frac{q_0+\cdots+q_n}{dr}\right\rceil-(N-n)=n+1-\left\lceil\frac{q_0+\cdots+q_n}{dr}\right\rceil.$$

If $\bfq=\bfone$, we set $\bfb=\rho\bfe_0$ and see that equality holds for $\proj^n$.
\end{proof}


A $\Z$--graded quotient ring $R$ of $S$ by a homogeneous ideal is called \emph{Gorenstein} if its localization $R_\mfrak$ to the irrelevant ideal of $S$ has finite injective dimension. An equivalent condition is that the $\Z$--graded free resolution of $R$ is self-dual up to a degree shift~\cite{MR1732040}. 
Indeed, a stronger result of Stanley states that $R$ is Gorenstein if and only if the $\Z$--graded Betti numbers of $R$ are symmetric in the following sense: $\beta_{i,j}=\beta_{p-i,\delta-j}$ where $p$ and $\delta$ are the projective dimension and regularity of $R$ respectively~\cite{MR0485835}. 

\begin{proposition}\label{gorenstein}
The module $R_{\bfq,d}$ is Gorenstein if $dr|q_0+\cdots+q_n$.
\end{proposition}

\begin{proof}
If $\bfc\in\Z\Acal$, then we know $dr|\bfq\cdot\bfc$. Then using~(\ref{qccheck}),
\begin{equation*}
\bfq\cdot\check\bfc\equiv-q_0-\cdots-q_n\;(\mathrm{mod}\;dr).
\end{equation*}
Then if $dr|q_0+\cdots+q_n$, then $dr|\bfq\cdot\bfc$ if and only if $dr|\bfq\cdot\check\bfc$, so $\bfc\in\Z\Acal$ if and only if $\check\bfc\in\Z\Acal$. Therefore, using Theorems~\ref{bhtheorem} and~\ref{dualformula}, we can say
$$\beta_{i,\bfc}(R_{\bfq,d})=\beta_{N-n-i,\check\bfc}(R_{\bfq,d}).$$
Indeed, in this case, we can recover the necessary symmetry of the $\Z$--graded Betti numbers.
\end{proof}

Our last theorem puts together some of these results to give an algorithm for finding the highest syzygy of $R_{\bfq,d}$.

\begin{theorem}\label{introlastsyzygy}
The highest nonzero syzygy of $R_{\bfq,d}$ is the $(N-n)$th syzygy. It is free with one generator in each multidegree $\bfc\in\N^{n+1}$ satisfying \begin{enumerate}
\item $dr|\bfq\cdot\bfc$, and
\item $\check\bfc\in\N^{n+1}$ with $\check\bfc\ngeq\bfa$ for all $\bfa\in\Acal$.
\end{enumerate}
In the case of $\proj^n$, when $\bfq=\bfone$, the $(N-n)$th syzygy of $R_{\bfq,d}$ is free of rank
$${dn-d\delta+d-1\choose n},\quad \delta=n+1-\left\lceil\frac{n+1}{d}\right\rceil$$
with one generator in each multidegree $\bfc\in\Z^{n+1}$ with $\check\bfc\in\N^{n+1}$ and 
\begin{equation}\label{remainder}
\bfone\cdot\check\bfc=d\left\lceil\frac{n+1}{d}\right\rceil-n-1.
\end{equation}
\end{theorem}

\begin{proof}
That the $(N-n)$th syzygy is the highest nonzero syzygy is the content of Proposition~\ref{pdim}. It is free precisely because it is the highest syzygy. We know that $dr|\bfq\cdot\bfc$ if and only if $\bfc\in\Z\Acal$, and $\beta_{N-n,\bfc}=\dim\tilde H_{-1}(\pilevar_{\check\bfc})$. Hence the first statement follows from~(\ref{reducedfact}).

Now let us apply this to the case of $\bfq=\bfone$. Here, the first condition is equivalent to $\bfone\cdot\check\bfc\equiv -n-1\;(\mathrm{mod}\;d)$ by~(\ref{qccheck}), and assuming $\check\bfc\in\N^{n+1}$, the second condition is equivalent to $\bfone\cdot\check\bfc<d$. Hence both conditions are satisfied if and only if $\check\bfc\in\N^{n+1}$ and~(\ref{remainder}) holds.

Counting the set of $\check\bfc$ satisfying this condition yields the desired result.
\end{proof}

Notice that in the case of a Veronese module for $\proj^n$, the highest syzygy is generated in a single $\Z$--graded degree. Furthermore, Proposition~\ref{reg} tells us that in this case, that degree is precisely $N+1-\left\lceil\frac{q_0+\cdots+q_n}{dr}\right\rceil$. One might be tempted to conjecture that the highest syzygy is always generated in a single $\Z$--graded degree, or alternatively that equality always holds in Proposition~\ref{reg}. Both conjectures turn out to be false as illustrated by the next two examples.

\begin{example}
Consider the first Veronese embedding of $\proj(3,3,3,2,2,2,2)$. Using the standard defining polytope, we have
$$\Acal=\left\{(a_1,a_2,a_3,a_1',a_2',a_3',a_4')\in\N^7\left|\begin{array}{c}\text{$\sum a_i=2$, and $a_i'=0$ for all $i$;}\\ \text{or $a_i=0$ for all $i$, and $\sum a_i'=3$}\end{array}\right.\right\}.$$

Here, $r=6$, so if $\bfc\in\Z\Acal$, then
$$\bfq\cdot\check\bfc\equiv -3-3-3-2-2-2-2\equiv1\;(\mathrm{mod}\;6)$$
If equality were to hold in~(\ref{regbound}), the regularity would be $4$ and there would need to be $\check\bfc\in\N^7$ such that $\bfq\cdot\check\bfc=1$. However, no such $\check\bfc$ exists. In this example $N=25$ and so the projective dimension is $19$. Applying Theorem~\ref{introlastsyzygy}, we can calculate that the $19$th syzygy is free of rank $30$ and generated in $\Z$--graded degree $22$, and so the regularity is $3$. 
\end{example}

\begin{example}
Consider the first Veronese embedding of $\proj(3,2,1,1,1,1,1,1)$. Here
$$\Acal=\left\{\begin{array}{l}
(2,0,0,0,0,0,0,0),\\
(1,1,a_1',a_2',a_3',a_4',a_5',a_6') \text{ with $\sum a_i'=1$},\\
(1,0,a_1',a_2',a_3',a_4',a_5',a_6') \text{ with $\sum a_i'=3$},\\
(0,3,0,0,0,0,0,0),\\
(0,2,a_1',a_2',a_3',a_4',a_5',a_6') \text{ with $\sum a_i'=2$},\\
(0,1,a_1',a_2',a_3',a_4',a_5',a_6') \text{ with $\sum a_i'=4$,}\\
(0,0,a_1',a_2',a_3',a_4',a_5',a_6') \text{ with $\sum a_i'=6$}
\end{array}\right\}$$

Here $r=6$, $N=1+6+56+1+21+126+462-1=672$, and the sum of the weights is $11$. Hence if $\bfc\in\N\Acal$, $\bfq\cdot\check\bfc\equiv1\;(\mathrm{mod}\;6)$. The only points $\check\bfc$ with $\bfq\cdot\check\bfc=1$ are $\bfe_2,\ldots,\bfe_7$. However, the point $\check\bfc=(1,2,0,0,0,0,0,0)$ has $\bfq\cdot\check\bfc=7$ yet $\check\bfc\ngeq\bfa$ for all $\bfa\in\Acal$.

Therefore, in this example, the highest syzygy has six generators of $\Z$--graded degree $671$ and one in $\Z$--graded degree $670$. The regularity is $6$, so equality holds in (\ref{regbound}), but the highest syzygy is not generated in a single total degree.
\end{example}

We conclude with the observation that the following classical combinatorial theorem may be helpful in applying Theorem~\ref{introlastsyzygy}.

\begin{theorem}[Euler]
Given $\bfq\in\N^{n+1}$ and $k\geq0$, the number $T_\bfq(k)$ of points $\bfc\in\N^{n+1}$ such that $\bfq\cdot\bfc=k$ is equal to the coefficient on $X^k$ in the power series expansion of
$$S_\bfq(X):=\prod_{i=0}^n\frac{1}{1-X^{q_i}}.$$
In particular,
$$T_\bfq(k)=\frac{1}{2\pi i}\oint\frac{x^{-k-1}}{\prod_{i=0}^n1-x^{q_i}}dx$$
where the residue is taken around the origin.
\end{theorem}

For a proof, see Section~12.1 of~\cite{MR2722776}.

\bibliographystyle{amsplain}
\providecommand{\bysame}{\leavevmode\hbox to3em{\hrulefill}\thinspace}
\providecommand{\MR}{\relax\ifhmode\unskip\space\fi MR }
\providecommand{\MRhref}[2]{%
  \href{http://www.ams.org/mathscinet-getitem?mr=#1}{#2}
}
\providecommand{\href}[2]{#2}

\end{document}